\documentclass[11pt]{amsart}

\usepackage{epsfig,overpic,comment}
\usepackage[usenames,dvipsnames,svgnames,table]{xcolor}
\usepackage[hyphens]{url}
\usepackage[pagebackref,linktocpage=true,colorlinks=true,linkcolor=Blue,citecolor=BrickRed,urlcolor=RoyalBlue]{hyperref}
\usepackage[msc-links,abbrev]{amsrefs}
\usepackage{amsmath,amsthm,amssymb,marginnote,upgreek}
\usepackage{enumerate}

\newtheorem{theorem}{Theorem}[section]
\newtheorem{corollary}[theorem]{Corollary}
\newtheorem{lemma}[theorem]{Lemma}
\newtheorem{proposition}[theorem]{Proposition}

\theoremstyle{definition}
\newtheorem{note}[theorem]{Note}

\newcommand{\be}{\begin{equation}}
\newcommand{\ee}{\end{equation}}

\newcommand{\R}{\mathbf{R}}

\newcommand{\C}{\mathcal{C}}
\newcommand{\G}{\mathcal{G}}

\renewcommand{\H}{\mathbf{H}}
\renewcommand{\S}{\mathbf{S}}
\renewcommand{\epsilon}{\varepsilon}

\DeclareMathOperator{\Pf}{Pf}
\DeclareMathOperator{\TPf}{TPf}

\begin{document}

\title[Isoperimetric and total curvature inequalities]{Isoperimetric and total curvature inequalities in\\Cartan-Hadamard manifolds with  nullity}

\author{Mohammad Ghomi}
\address{School of Mathematics, Georgia Institute of Technology,
Atlanta, GA 30332}
\email{ghomi@math.gatech.edu}
\urladdr{ghomi.math.gatech.edu}

\begin{abstract}
Using the Chern-Gauss-Bonnet theorem, we establish a sharp inequality for the total Gauss-Kronecker curvature of convex hypersurfaces in Cartan-Hadamard manifolds $M^n$ with nullity index at least $n-3$. 
Consequently, the Euclidean isoperimetric inequality extends to $M^n$, which proves the Cartan-Hadamard conjecture for these spaces.
\end{abstract}

\date{\today\,(Last Typeset)}

\makeatletter
\def\subjclassname{\textup{2020} Mathematics Subject Classification}
\makeatother


\thanks{The author was supported by NSF grant DMS-2202337.}

\maketitle

\section{Introduction}
A \emph{Cartan-Hadamard manifold} $M^n$ is a complete simply connected Riemannian space of nonpositive sectional curvature. These spaces are natural generalizations of Euclidean space $\R^n$. The \emph{nullity index} $\mu(M)$, introduced by Chern and Kuiper
\cite{chern-kuiper1952}, measures the flatness of $M$. More precisely, at each point $p\in M$, let
\[
        \mathcal N_p
        :=
       \big \{Z\in T_pM\mid R(X,Y)Z=0
        \text{ for all }X,Y\in T_pM\big\}
\]
be the nullity space of the curvature  tensor $R$ of $M$. Then 
\(
       \mu(M):= \inf_{p\in M}\dim (\mathcal N_p).
\)
So $M^n=\R^n$ if and only if $\mu(M)=n$.
A \emph{convex hypersurface} $\Gamma\subset M$ is the boundary of a compact convex set with nonempty interior. The \emph{total curvature} of $\Gamma$, when it is $\C^{1,1}$, is defined as $\G(\Gamma):=\int_\Gamma GK$ where $GK$ is the Gauss-Kronecker curvature. Let $|\S^{n-1}|$ denote the volume of the unit sphere in $\R^n$. We show

\begin{theorem}\label{thm:main}
Let
\(
        M^n
\)
be a Cartan-Hadamard manifold with $\mu(M)\geq n-3$. Then 
\be\label{eq:G}
        \G(\Gamma)\geq |\S^{n-1}|,
\ee
for any smooth $(\C^\infty)$ convex hypersurface $\Gamma\subset M$. 
\end{theorem}

 In any Cartan-Hadamard manifold where the total curvature inequality \eqref{eq:G} holds, the classical Euclidean isoperimetric inequality follows  \cite[Thm. 7.1]{ghomi-spruck2022}. Specifically, letting $B^n$ denote the unit ball in $\R^n$, we obtain

\begin{corollary}
For any bounded region $\Omega\subset M^n$ of finite perimeter,
\be\label{eq:Isop}
\frac{|\partial\Omega|^n}{\,\;\;|\Omega|^{n-1}}\geq \frac{|\S^{n-1}|^n}{\;\;|B^n|^{n-1}},
\ee
with equality only if $\Omega$ is isometric to a ball in $\R^n$.
\end{corollary}

The \emph{Cartan-Hadamard conjecture} \cite{gromov1999, aubin1975, burago-zalgaller1988} states that the isoperimetric inequality \eqref{eq:Isop} holds in all Cartan-Hadamard manifolds. This has been established in dimensions $n\leq 4$ \cite{weil1926,kleiner1992,croke1984}, and recently in dimensions $n\leq 9$ \cite{chen-ghomi-wang2026a, chen-ghomi-wang2026b}, while the total curvature inequality \eqref{eq:G} is known only for $n\leq 3$. Both inequalities  hold in all dimensions when $M^n$ has constant curvature or $\Gamma$ is a geodesic sphere \cite{ghomi-spruck2022,borbely2002}. Theorem \ref{thm:main} gives the first class of Cartan-Hadamard
manifolds of dimension \(n\geq4\), beyond space forms, for which \eqref{eq:G} holds, and also yields irreducible nonhomogeneous examples of dimension $n\geq 10$ where \eqref{eq:Isop} holds.
See \cite{kloeckner-kuperberg2019,ghomi-spruck2022} for background and \cite{ghomi-spruck2023b,ghomi-stavroulakis2026, ghomi-stavroulakis2026b,li-lin-xu2025} for more recent results.

The isoperimetric inequality is stable under Cartesian products \cite[p. 335]{gromov1991},  \cite[Sec. 3.3]{ros2005}, but the total curvature inequality \eqref{eq:G} is new even in the case of product manifolds, such as $\H^3\times\R$ and $\H^2\times\R^2$, where $\H^n$ is  hyperbolic $n$-space (but Theorem \ref{thm:main} does not apply to $\H^4\times\R$, for instance, where the nullity is $1<5-3$). Manifolds with nullity arise naturally in connection with isometric immersions, totally geodesic
foliations, and rigidity phenomena in Riemannian geometry
\cite{ferus1971,rosenthal1969,gorodski-guimaraes2023}. There are many examples which do not split off a Euclidean factor, even locally.  For conullity-two examples, see
\cite{brooks2025,vanhook2024,boeckx-kowalski-vanhecke1996}; for related
graph manifolds, whose universal covers give irreducible Cartan-Hadamard
manifolds with nullity, see \cite{florit-ziller2020}.

The proof of Theorem \ref{thm:main} is an application of the
Chern-Gauss-Bonnet formula \eqref{eq:GBC} together with Gauss' equation \eqref{eq:Gauss}.
When \(n\) is odd, we apply the formula to the even-dimensional
hypersurface \(\Gamma\), and when \(n\) is even,  to the
convex body $C$ bounded by \(\Gamma\). In either case, the nullity
assumption implies that all terms in the Chern-Gauss-Bonnet integrands
involving two or more ambient curvature factors vanish, i.e., the
integrands become affine functions of the ambient curvature. Thus
they reduce to the Gauss-Kronecker curvature \(GK\), up to a
correction term which contains exactly one ambient sectional curvature
factor and some principal curvatures of \(\Gamma\). By the Cartan-Hadamard
assumption, the ambient curvature factor is nonpositive, while
convexity of \(\Gamma\) ensures that the principal curvatures are nonnegative.
Hence the correction term is nonpositive, and we obtain \eqref{eq:G}.

\section{The Chern-Gauss-Bonnet Theorem}\label{sec:gbc}
We begin by recording the formulation of the Chern-Gauss-Bonnet theorem  that we need.
 Let $N^k$ be an oriented even-dimensional Riemannian manifold with metric $\langle \cdot,\cdot\rangle$, Levi-Civita connection $\nabla$, and curvature tensors
 \[
 R(X,Y)Z:=\nabla_X\nabla_Y Z-\nabla_Y\nabla_X Z-\nabla_{[X,Y]}Z,
 \]
\(
        R(X,Y,Z,W):=\langle R(X,Y)Z,W\rangle.
\)
At each point $p\in N$,  choose a positively oriented orthonormal basis $e_1,\dots,e_k\in T_pN$. The corresponding curvature $2$-forms of $N$ are
\[
        \Omega_{ij}(X,Y):=R(X,Y,e_j,e_i).
\]
The Pfaffian $k$-form associated to  the skew-symmetric matrix $\Omega=(\Omega_{ij})$ is defined by
\[
\Phi:= \sum_{\sigma\in S_{k}}\operatorname{sgn}(\sigma)\,
        \Omega_{\sigma(1)\sigma(2)}\wedge\cdots\wedge
        \Omega_{\sigma(k-1)\sigma(k)},
\]
where \(S_k\) denotes the symmetric group on \(k\) elements,
\(\operatorname{sgn}(\sigma)=1\) for even permutations, and
\(\operatorname{sgn}(\sigma)=-1\) for odd permutations.
The associated normalized scalar function is
\be\label{eq:Pf}
        \Pf_N:=c_k\Phi(e_1,\dots,e_{k}),
\ee
where the constant $c_k:=1/k!$ is chosen so that $\Pf_{\S^{k}}\equiv 1$. Note that $\Pf_N$ is independent of the choice of the frame $e_i$, since $\Phi$ is invariant under conjugation by \(\textup{SO}(k)\).

Now suppose that $N$ is compact with boundary $\partial N$. Let $\nu=e_{k}$ be the outward unit normal of $\partial N$, and let $e_1,\dots,e_{k-1}$ be a positively oriented local orthonormal frame on $\partial N$. For \(X\in T(\partial N)\), let
\(
        A(X):=\nabla_X\nu
\)
be the shape operator, and
\[
        \alpha_i(X):=\langle A(X),e_i\rangle
\]
be the corresponding $1$-forms on $\partial N$. So if $e_i$ are the principal directions, which diagonalize $A$, then $\alpha_i(e_i)=\kappa_i$ are the corresponding principal curvatures of $\partial N$. The Chern boundary term is given by the transgression formula
\be\label{eq:TPf}
        \TPf_{\partial N}:=\sum_{s=0}^{k/2-1} c_{k,s}\Phi_s(e_1,\dots,e_{k-1}),
\ee
where $c_{k,s}:={|\S^{k-1}|}/
({(4\pi)^s\, s!\,|\S^{k-1-2s}|\, (k-1-2s)!})$ are normalizing constants, and $\Phi_s$ are $(k-1)$-forms given by combining the shape operator forms of $\partial N$ and the ambient curvature forms of $N$ as follows
\begin{multline*}
        \Phi_s:=
        \sum_{\sigma\in S_{k-1}}\operatorname{sgn}(\sigma)
        \bigl(\alpha_{\sigma(1)}\wedge\cdots\wedge
        \alpha_{\sigma(k-1-2s)}\bigr) \\
        \wedge
        \bigl(\Omega_{\sigma(k-2s)\sigma(k-2s+1)}\wedge\cdots\wedge
        \Omega_{\sigma(k-2)\sigma(k-1)}\bigr).
\end{multline*}
Here the second parenthesis is omitted when $s=0$. Thus the first term in \eqref{eq:TPf} reduces to the Gauss-Kronecker curvature:
\be\label{eq:GK}
c_{k,0}\Phi_0(e_1,\dots, e_{k-1})=\det(A)=\kappa_1\dots\kappa_{k-1}=GK.
\ee
Let $\chi$ denote the Euler characteristic. We have

\begin{theorem}[Chern-Gauss-Bonnet]\label{thm:gbc}
Let $N^{k}$ be a compact oriented even-dimensional Riemannian manifold, with possibly empty smooth boundary $\partial N$. Then
\be\label{eq:GBC}
        \chi(N)=
        \frac{2}{|\S^{k}|}\int_N \Pf_N
        +\frac{1}{|\S^{k-1}|}\int_{\partial N} \TPf_{\partial N},
\ee
where the boundary term is omitted when $\partial N=\emptyset$.
\end{theorem}

For proofs, see Chern \cite{chern1945}, Allendoerfer-Weil \cite{allendoerfer-weil1943}, or Spivak \cite[Addendum 2]{spivak:v5}. Our formulation here is consistent with Morgan-Johnson \cite[Sec. 4.1]{morgan-johnson2000} and Gilkey-Park \cite[Thm. 1.3]{gilkey-park2015}. It quickly follows that

\begin{corollary}\label{cor:gbc}
Let $M^n$ be a Cartan-Hadamard manifold, and $\Gamma\subset M$ be a smooth convex hypersurface bounding a convex body $C$. If $n$ is odd, then
\[
\int_\Gamma \Pf_\Gamma=|\S^{n-1}|.
\]
If $n$ is even, then
\[
\int_\Gamma \TPf_{\Gamma}=|\S^{n-1}|-2\frac{|\S^{n-1}|}{|\S^n|}\int_C\Pf_C.
\]
\end{corollary}
\begin{proof}
Since $M^n$ is Cartan-Hadamard, $C$ is homeomorphic to the unit ball $B^n$, and consequently $\Gamma$ is homeomorphic to $\S^{n-1}$. In particular, when $n$ is odd, $\chi(\Gamma)=2$ which yields the first equality above by Theorem \ref{thm:gbc} with $N=\Gamma$. The second equality also follows immediately from Theorem \ref{thm:gbc} with $N=C$, since  $\chi(C)=1$.
\end{proof}

\section{Proof of Theorem \ref{thm:main}}
The nullity assumption has the following consequence for the curvature forms:
\begin{lemma}\label{lem:curvature-forms-product}
If \(\mu(M)\geq n-3\), then for all indices \(1\leq i,j,k,\ell\leq n\),
\be\label{eq:Omega}
        \Omega_{ij}\wedge\Omega_{k\ell}=0,
\ee
with respect to any orthonormal basis.
\end{lemma}

\begin{proof}
Fix \(p\in M\), and let
\(
        E_p:=\mathcal N_p^\perp .
\)
By hypothesis, \(\dim(E_p)\leq3\). Let \(\pi  \colon T_pM\to E_p\) be the
orthogonal projection. Then
\[
        R(X,Y,Z,W)=R(\pi  X,\pi  Y,\pi  Z,\pi  W)
\]
for all \(X,Y,Z,W\in T_pM\). Indeed, if one of these vectors lies in
\(\mathcal N_p\), then $R$ vanishes. This is evident from the definition of $\mathcal{N}_p$ when the  vector appears in the
third slot, and follows for the other slots by the symmetries of
\(R\).
Consequently,
\[
        \Omega_{ij}(X,Y)
        =
        R(\pi  X,\pi  Y,\pi  e_j,\pi  e_i).
\]
Thus, as a \(2\)-form in \(X,Y\), each \(\Omega_{ij}\) is pulled back
from the vector space \(E_p\). It follows that
\(
        \Omega_{ij}\wedge\Omega_{k\ell}
\)
is pulled back from a \(4\)-form on \(E_p\). Since \(\dim E_p\leq3\),
every \(4\)-form on \(E_p\) vanishes, which completes the proof since $p$ is arbitrary.
\end{proof}

We orient
\(\Gamma\) by the outward unit normal \(\nu\), and use the convention
\(
        A(V)=\nabla_V\nu
\)
for the shape operator.
Then the principal curvatures \(\kappa_1,\dots,\kappa_{n-1}\) of
\(\Gamma\), and the Gauss-Kronecker curvature
\(
        GK=\kappa_1\cdots\kappa_{n-1}
\)
are nonnegative.
At each point \(p\in\Gamma\), choose an orthonormal basis
\(e_1,\dots,e_{n-1}\in T_p\Gamma\) which diagonalizes \(A\).
Let
\[
        K_{ij}:=\Omega_{ij}(e_i,e_j)
\]
be the sectional curvature of $M$ with respect to the plane spanned
by \(e_i,e_j\). The Gauss
equation gives
$
        R^\Gamma_{ijji}=K_{ij}+\kappa_i\kappa_j.
$
Letting \(\theta_1,\dots,\theta_{n-1}\) be the coframe dual to
\(e_1,\dots,e_{n-1}\), the Gauss equation may be written as
\be\label{eq:Gauss}
        \Omega^\Gamma_{ij}
        =
        \Omega_{ij}
        +
        \kappa_i\kappa_j\,\theta_i\wedge\theta_j.
\ee
Substituting this equation in the expressions for the Pfaffians in the last section, we obtain the following result. Recall that \(C\) denotes the  convex body bounded by \(\Gamma\). 

\begin{proposition}\label{prop:integrand-forms}
If \(n\) is odd, then
\[
        \Pf_\Gamma \leq GK.
\]
If \(n\) is even and \(n\geq4\), then
\[
        \Pf_C\equiv0, \qquad\text{and}\qquad
        \TPf_{\Gamma}\leq GK.
\]
\end{proposition}

\begin{proof}
First suppose that \(n\) is odd. Then \(\Gamma\) has even dimension, and \(\Pf_\Gamma\)  is a linear
combination of wedge products of the forms \(\Omega^\Gamma_{ij}\) by \eqref{eq:Pf}.
After substituting the Gauss equation \eqref{eq:Gauss}, each term contains some number of
ambient curvature forms \(\Omega_{ij}\) and some number of factors
\(\kappa_i\kappa_j\,\theta_i\wedge\theta_j\). By Lemma
\ref{lem:curvature-forms-product}, all terms with two or more ambient
curvature forms vanish. Thus only the terms with no ambient curvature
form or exactly one ambient curvature form can contribute.

The contribution to \(\Pf_\Gamma\) from terms with no ambient curvature
form is precisely \(GK\).
Now consider a term with exactly one ambient curvature form, say \(\Omega_{ij}\). Since all the remaining factors come
from the diagonal second fundamental form, they contribute multiples of
\(\theta_a\wedge\theta_b\) and supply all coframe directions except
\(\theta_i,\theta_j\). Consequently, when the resulting \((n-1)\)-form
is evaluated on \(e_1,\dots,e_{n-1}\), we obtain
\[
        \Omega_{ij}(e_i,e_j)
        \prod_{\ell\neq i,j}\kappa_\ell 
        =
        K_{ij}
        \prod_{\ell\neq i,j}\kappa_\ell \leq 0,
\]
where the last inequality holds since $K_{ij}\leq 0$ by the Cartan-Hadamard assumption,  and $\kappa_\ell\geq 0$ by convexity of $\Gamma$.
Hence, we obtain
\be\label{eq:PfGamma}
        \Pf_\Gamma
        =
        GK+
        \frac{1}{n-2}
        \sum_{1\leq i<j\leq n-1}
        K_{ij}\prod_{\ell\neq i,j}\kappa_\ell \leq GK.
\ee

Next suppose that \(n\) is even and \(n\geq4\). By \eqref{eq:Pf},
\(\Pf_C\) consists of wedge products of \(n/2\) ambient curvature forms.
Since \(n/2\geq2\), each such term contains at least two ambient
curvature forms, and hence vanishes by Lemma
\ref{lem:curvature-forms-product}. So $\Pf_C\equiv 0$. 

Finally, to evaluate
\(\TPf_{\Gamma}\), note that in \eqref{eq:TPf}, the factors $\alpha_i$ coming
from the second fundamental form are the \(1\)-forms
\(
        \kappa_i\theta_i.
\)
 The first term, which contains no ambient curvature form, contributes the Gauss-Kronecker curvature  $GK$ by  \eqref{eq:GK}. As above, Lemma
\ref{lem:curvature-forms-product} eliminates all terms with two or more
ambient curvature forms. If a term contains exactly one ambient
curvature form \(\Omega_{ij}\), then again the remaining second fundamental
form factors supply all coframe directions except \(\theta_i,\theta_j\).
Hence only the component \(\Omega_{ij}(e_i,e_j)=K_{ij}\) contributes.
Thus, we have
\be\label{eq:TPfGamma}
        \TPf_{\Gamma}
        =
        GK+
        \frac{1}{n-2}
        \sum_{1\leq i<j\leq n-1}
        K_{ij}\prod_{\ell\neq i,j}\kappa_\ell \leq GK,
\ee
which completes the proof.
\end{proof}

Now if \(n\) is odd, then Corollary
\ref{cor:gbc} together with Proposition \ref{prop:integrand-forms} gives
\be\label{eq:last1}
        |\S^{n-1}|
        =
        \int_\Gamma \Pf_\Gamma
        \leq
        \int_\Gamma GK
        =
        \G(\Gamma).
\ee
If \(n=2\), then the ordinary Gauss-Bonnet theorem gives
\(2\pi=\int_C K+\int_\Gamma GK\), and since \(K\leq0\), it follows that
\(\G(\Gamma)=\int_\Gamma GK\geq2\pi=|\S^1|\). Finally suppose that \(n\geq4\) is
even. Then Corollary
\ref{cor:gbc} and Proposition \ref{prop:integrand-forms} give
\be\label{eq:last2}
        |\S^{n-1}|
        =
        \int_\Gamma \TPf_{\Gamma} 
        \leq
        \int_\Gamma GK
        =
        \G(\Gamma),
\ee
which completes the proof.

\section{Notes}
\begin{note}
The nullity assumption $\mu(M)\geq n-3$ was used in the proof of Theorem \ref{thm:main} only to derive  \eqref{eq:Omega}, which is a formally weaker condition at each point. Thus the total curvature inequality \eqref{eq:G} holds whenever  \eqref{eq:Omega} holds on $M$. More specifically, it is enough that  \eqref{eq:Omega} holds on $\Gamma$ when $n$ is odd, and on the convex body $C$ when $n$ is even.
\end{note}

\begin{note}
If equality holds in the total curvature inequality \eqref{eq:G}, for $n\geq 3$, then
equality holds in the corresponding estimate \eqref{eq:last1} or
\eqref{eq:last2}, according as \(n\) is odd or even. Hence equality also
holds in the corresponding pointwise inequality \eqref{eq:PfGamma} or
\eqref{eq:TPfGamma}, which yields
\[
        K_{ij}\prod_{\ell\neq i,j}\kappa_\ell=0
\]
for every \(i<j\). In particular, when \(\Gamma\) is strictly convex,
i.e., \(\kappa_\ell>0\) for all \(\ell\), equality in \eqref{eq:G}
forces the sectional curvatures of \(M\) to vanish on tangent planes of
\(\Gamma\), which in turn yields that \(C\) is flat
\cite{ghomi2025-convexity}.
\end{note}

\begin{note}
The total curvature of an arbitrary convex hypersurface
\(\Gamma\subset M^n\) may be defined as the limit of the total
curvatures of its outer parallel hypersurfaces, which are
\(\C^{1,1}\) \cite{ghomi-spruck2022}. With this definition, the total curvature functional
\(\mathcal G\) is continuous on the space of convex hypersurfaces in
Cartan-Hadamard manifolds, with respect to Hausdorff distance
\cite{ghomi2026-continuity}. Since smooth convex hypersurfaces are
dense in this space \cite[Lem. 3.9]{ghomi2026-total}, Theorem
\ref{thm:main} extends to all convex hypersurfaces, with no smoothness
assumption.
\end{note}

\section*{Acknowledgment}
The author thanks Matteo Raffaelli  and John Stavroulakis for useful communications.

\bibliography{references}

\end{document}